\documentclass[a4paper,twoside,11pt]{article}
\usepackage[utf8]{inputenc}
\usepackage{graphicx,epstopdf}
\usepackage{amsmath}
\usepackage{amsthm}
\usepackage{amssymb}
\usepackage{color}
\usepackage{mathtools}
\usepackage{subfig}
\usepackage{siunitx}
\usepackage{amsthm}
\usepackage[font=footnotesize,labelfont=bf]{caption}

\usepackage{geometry,calc}
\usepackage{listings}

\newtheorem{theorem}{Theorem}[section]

\newtheorem{lemma}[theorem]{Lemma}
\theoremstyle{remark}
\newtheorem{remark}{Remark}

\usepackage[most]{tcolorbox}
\definecolor{darkgreen}{rgb}{0.0, 0.55, 0.0}

\title{Stochastic Production Planning with Regime Switching: Numerical and
Sensitivity Analysis, Optimal Control, and Python Implementation}
\author{Dragos-Patru Covei\thanks{Department of Applied Mathematics, The Bucharest University of Economic Studies, Piata Romana, No. 6, Bucharest, District 1, 010374, Romania}}

\date{}

\begin{document}
\maketitle
	
\begin{abstract}
\noindent
This study investigates a stochastic production planning problem with
regime-switching parameters, inspired by economic cycles impacting
production and inventory costs. The model considers types of goods and
employs a Markov chain to capture probabilistic regime transitions, coupled
with a multidimensional Brownian motion representing stochastic demand
dynamics. The production and inventory cost optimization problem is
formulated as a quadratic cost functional, with the solution characterized
by a regime-dependent system of elliptic partial differential equations
(PDEs). Numerical solutions to the PDE system are computed using a monotone
iteration algorithm, enabling quantitative analysis. Sensitivity analysis
and model risk evaluation illustrate the effects of regime-dependent
volatility, holding costs, and discount factors, revealing the conservative
bias of regime-switching models when compared to static alternatives.
Practical implications include optimizing production strategies under
fluctuating economic conditions and exploring future extensions such as
correlated Brownian dynamics, non-quadratic cost functions, and geometric
inventory frameworks. This research bridges the gap between theoretical
modeling and practical applications, offering a robust framework for dynamic
production planning.

\medskip\noindent
\textbf{AMS subject classification}: 49K20; 49K30; 90C31;90C31;90B30; 90C31; 90B30.
		
\medskip\noindent
\textbf{Keywords}: Quadratic cost functional; Sensitivity analysis; Optimizing production strategies
\end{abstract}

\section{Introduction}

Stochastic production planning has been extensively studied across various
applications and methodologies. Below is a synthesis of some significant
contributions from the literature relevant to this study: Bensoussan et al. 
\cite{BS} (addressed stochastic production planning with constraints, laying
the groundwork for production optimization under uncertainty), Cadenillas et
al. \cite{CLP} (explored optimal production management considering demand
variations influenced by business cycles, introducing regime-switching
dynamics to production planning), Cadenillas et al. \cite{CS} (presented
strategies for production management under regime switching with production
constraints, showcasing advanced mathematical frameworks), Dong et al. \cite%
{DMDK} (applied optimal production control theory to energy management in
microgrids, highlighting regime-switching dynamics in engineering systems),
Gharbi et al. \cite{GK} (investigated stochastic production control in
multi-product, multi-machine manufacturing systems, emphasizing inventory
optimization under dynamic conditions), Covei et al. \cite{CP} (introduced
elliptic PDEs within stochastic frameworks, contributing to the mathematical
modeling of production strategies), the paper \cite{CCP} (developed
stochastic production planning models with regime switching, offering
insights into inventory management under uncertainty), Ghosh et al. \cite%
{GAS} (focused on optimal control for switching diffusions, providing robust
methodologies applicable to flexible manufacturing systems), the author \cite%
{CDPCM} (analyzed parabolic PDEs for production planning problems, offering
valuable perspectives on dynamic system modeling).

These references form the basis of our research, guiding the mathematical
formulation and solution approach detailed in the subsequent sections.

In this study, we consider a stochastic production planning problem with
regime-switching parameters, motivated by scenarios where economic cycles
influence production and inventory costs. Regime-switching modeling, widely
applied across fields such as financial economics, civil engineering, and
manufacturing systems, provides robust tools for analyzing systems governed
by multiple dynamic regimes. Building on the existing body of work, this
paper presents a mathematical framework that defines optimal production
strategies using elliptic partial differential equations (PDEs). It serves
as a continuation of the study conducted in \cite{CCP}, which, however,
lacks any practical implementation of its findings.

Regime switching reflects scenarios where system characteristics change
depending on distinct states, such as economic growth vs. recession or high
vs. low demand periods. Recent studies in civil engineering have analyzed
optimal control problems for energy systems with peak and off-peak demand
regimes, while manufacturing systems have examined production strategies for
multiple machines subject to breakdowns. Among notable contributions is the
use of value function approaches with solutions characterized via elliptic
PDE systems. However, few works have investigated regime-switching
production planning in inventory systems, especially under stochastic
dynamics.

This paper addresses a stochastic production planning problem involving $%
N\geq 1$ types of goods stored in inventory, aiming to minimize production
and inventory costs over time under regime-switching economic parameters.
The optimization criterion is based on a quadratic cost functional
representing production and holding costs adjusted for stochastic demand,
with a stopping criterion terminating production when inventory levels
exceed a specified threshold. Regimes are modeled using a Markov chain
capturing probabilistic transitions between states, coupled with an $N$%
-dimensional Brownian motion accounting for stochastic demand fluctuations.

The stochastic dynamics of inventory levels are governed by:%
\begin{equation}
dy_{i}(t,\epsilon (t))=p_{i}(t,\epsilon (t))dt+\sigma _{\epsilon
(t)}dw_{i}(t),\text{ }y_{i}(0,\epsilon (0))=y_{i}^{0,\epsilon (0)}\text{,}%
\quad i=1,\dots ,N,  \label{di}
\end{equation}%
where $p_{i}$ is the deterministic production rate, $y_{i}(t,\epsilon (t))$
is an It\^{o} process in $\mathbb{R}$, $\sigma _{\epsilon (t)}$ is the
regime-dependent volatility, and $\epsilon (t)$ is a Markov chain
representing economic regimes.

The cost functional is defined as:

\begin{equation*}
J(p_{1},\dots ,p_{N})=\mathbb{E}\int_{0}^{\tau }\left(
\sum_{i=1}^{N}p_{i}^{2}(t,\epsilon (t))+f_{\epsilon (t)}(y(t,\epsilon
(t)))\right) e^{-\alpha _{\epsilon (t)}t}\,dt,
\end{equation*}%
subject to (\ref{di}) and the stopping time $\tau $, when inventory exceeds
a threshold $R>0$.

The structure of this paper is as follows: Section \ref{1} introduces the
mathematical formulation of the model and its objectives. Section \ref{2}
explains the methodology, including the derivation of the HJB equations and
the existence of solutions. Section \ref{3} focuses on the optimal control
of the problem at hand. Section \ref{4} details the numerical algorithm for
the obtained solution. Section \ref{5} provides a discussion on sensitivity
analysis, model examination, and visualization of the results. Section \ref%
{7} proposes future research directions. Section \ref{8} concludes with the
final observations, and Section \ref{9} concludes with the Python
implementation of the numerical algorithm. 

\section{Theoretical Framework \label{1}}

In this section, we present the mathematical formulation of the stochastic
production planning problem with regime switching. The model incorporates
random demand, regime-dependent parameters, and production controls, as
described below.

\subsection{Regime Switching and Dynamics}

We consider a probability space $(\Omega ,\mathcal{F},P)$ together with a
standard $\mathbb{R}^{N}$-valued Brownian motion $w=\{w_{t}:t\geq 0\}$ and
an observable finite-state continuous-time homogeneous Markov chain $%
\epsilon \left( t\right) =\{\epsilon _{t}:t\geq 0\}$.

We denote by $\mathcal{F}=\{\mathcal{F}_{t}:t\geq 0\}$ the $P$-augmentation
of the filtration $\{\mathcal{F}(w,\epsilon )_{t}:t\geq 0\}$ generated by
the Brownian motion and the Markov chain, where

\begin{equation*}
\mathcal{F}(w,\varepsilon )_{t}:=\sigma (\{w_{s},\epsilon _{s}:0\leq s\leq
t\})
\end{equation*}%
for every $t\geq 0$.

The manager of a firm wants to control the inventory of a given item. We
assume a stochastic production environment driven by two sources of
randomness:

\begin{itemize}
\item \textbf{Markov Chain:} A continuous-time homogeneous Markov chain $%
\epsilon(t) $, with states $\{1, 2\} $, represents economic regimes. These
regimes may correspond to scenarios such as economic growth ($\epsilon(t) =
1 $) or recession ($\epsilon(t) = 2 $).

\item \textbf{Brownian Motion:} An $N $-dimensional Brownian motion $w(t) =
(w_1(t), \dots, w_N(t)) $ models stochastic demand fluctuations in inventory
levels.
\end{itemize}

We also assume that $\epsilon $ and $w$ are independent, and that the Markov
chain $\epsilon $ has a strongly irreducible generator which is given by:

\begin{equation*}
A=%
\begin{bmatrix}
-a_{1} & a_{1} \\ 
a_{2} & -a_{2}%
\end{bmatrix}%
,\text{ (transition rate matrix of the Markov chain }\epsilon (t)\text{)}
\end{equation*}%
where $a_{1}>0$ and $a_{2}>0$. In this case%
\begin{equation*}
\text{if }p_{t}\left( t\right) =E[\epsilon (t)]\in \mathbb{R}^{2}\quad \text{%
then}\quad \frac{dp_{t}\left( t\right) }{dt}=A\epsilon (t),
\end{equation*}%
and $\epsilon \left( t\right) $ is explicitly described by the integral form%
\begin{equation*}
\epsilon \left( t\right) =\epsilon \left( 0\right) +\int_{0}^{t}A\epsilon
\left( s\right) ds+M\left( t\right) ,
\end{equation*}%
where $M(t)$ is a martingale with respect to $\mathcal{F}$.

\subsection{Inventory Dynamics and State Variables}

Let $y_{i}(t)$ denote the inventory levels of good $i$ at time $t$, adjusted
for demand, and let $p_{i}(t,\epsilon (t))$ denote the production rate
(control variable) for good $i$ at time $t$ under regime $\epsilon (t)$. The
stochastic dynamics of the inventory are governed by:

\begin{equation}
dy_{i}(t)=p_{i}(t,\epsilon (t))dt+\sigma _{\epsilon (t)}dw_{i}(t),\text{ }%
y_{i}(0)=y_{i}^{0}\text{,}\quad i=1,\dots ,N,  \label{sd}
\end{equation}%
where: $\sigma _{\epsilon (t)}$ is the regime-dependent volatility, $%
w_{i}(t) $ is the $i$-th component of an $N$-dimensional Brownian motion, $%
\epsilon (t)$ is a Markov chain representing economic regimes and $y_{i}^{0}$
denote the initial inventory level of good $i$.

\subsection{Objective Function}

The objective of the stochastic production planning problem is to minimize
the total expected cost incurred over time. These costs include both
production costs and inventory holding costs, adjusted for regime-switching
dynamics and exponential discounting. This is formalized through the
following components:

\subsubsection*{Production Costs}

The cost associated with the production rate $p_i(t, \epsilon(t)) $ for good 
$i $ is quadratic and regime-dependent. The quadratic form ensures
tractability in optimization and is expressed as:

\begin{equation*}
C_{p}(t)=\sum_{i=1}^{N}p_{i}^{2}(t,\epsilon (t)),
\end{equation*}%
where $p_{i}(t,\epsilon (t))$ represents the net production rate (actual
production minus demand).

\subsubsection*{Inventory Costs}

The holding cost for storing the inventory is modeled as a convex function
of the inventory levels. It accounts for regime-switching parameters and is
given by:

\begin{equation*}
C_{h}(t)=f_{\epsilon (t)}(y(t,\epsilon (t))),
\end{equation*}%
where $f_{\epsilon (t)}(\cdot )$ represents regime-dependent holding costs.
The convexity of $f_{\epsilon (t)}(\cdot )$ reflects the increasing marginal
cost of holding excess inventory.

\subsubsection*{Discount Factor}

To account for the time value of money, the costs are exponentially
discounted with a regime-dependent discount rate $\alpha _{\epsilon (t)}$.
The discount factor ensures that costs incurred in the future are valued
less than those incurred immediately.

\subsubsection*{Cost Functional}

The factory aims to minimize production and inventory costs, subject to the
stochastic dynamics (\ref{sd}) described above. The total cost functional,
combining production costs, inventory costs, and exponential discounting, is
given by:

\begin{equation*}
J(p_{1},\dots ,p_{N})=\mathbb{E}\left[ \int_{0}^{\tau }\left(
\sum_{i=1}^{N}p_{i}^{2}(t,\epsilon (t))+f_{\epsilon (t)}(y(t,\epsilon
(t)))\right) e^{-\alpha _{\epsilon (t)}t}dt\right] ,
\end{equation*}%
where: $p_{i}(t,\epsilon (t))$ is the production rate for good $i$ at time $%
t $ under regime $\epsilon (t)$, $p_{i}^{2}(t,\epsilon (t))$ is the
quadratic production cost for good $i$, $\
\sum_{i=1}^{N}p_{i}^{2}(t,\epsilon (t))$ is the regime-dependent production
costs, modeled as quadratic functions of the production rate, $\
y(t,\epsilon (t))$ represents the inventory levels of goods, adjusted for
demand, $\ f_{\epsilon (t)}(y(t,\epsilon (t)))$: regime-dependent inventory
holding costs (holding cost, modeled as convex functions $f_{1}(x)$ and $%
f_{2}(x)$) and $\alpha _{\epsilon (t)}$: regime-dependent discount rate for
exponential discounting;

The stopping time $\tau $ is defined as the moment when the inventory
exceeds an exogenous threshold $R$, i.e.:

\begin{equation*}
\tau =\inf \{t>0:\Vert y(t,\epsilon (t))\Vert \geq R\}\text{, }y(t,\epsilon
(t))=\left( y_{1}(t,\epsilon (t)),...,y_{N}(t,\epsilon (t))\right) ,
\end{equation*}%
where $\left\Vert \circ \text{ }\right\Vert $ stands for the Euclidian norm.

\section{Optimization Problem \label{2}}

The primary objective of the stochastic production planning problem is to
minimize the total expected cost, which comprises both production and
inventory holding costs, subject to the constraints of stochastic inventory
dynamics and regime-switching parameters. This optimization problem is
formulated as follows.

\subsection{Optimization Objective}

The objective is to determine the optimal production rates $p_{1}(t,\epsilon
(t))$, ..., $p_{N}(t,\epsilon (t))$ that minimize the total cost functional $%
J$, while satisfying the stochastic inventory dynamics. Mathematically, this
is expressed as:

\begin{equation*}
\inf_{p_{1},\dots ,p_{N}}J(p_{1},\dots ,p_{N}),
\end{equation*}%
subject to the inventory dynamics:

\begin{equation*}
dy_i(t) = p_i(t, \epsilon(t)) dt + \sigma_{\epsilon(t)} dw_i(t), \quad
y_i(0) = y_i^0, \quad i = 1, \dots, N.
\end{equation*}

The optimization problem is solved over a finite horizon, up to the stopping
time $\tau $, and incorporates the effects of regime switching. The
constraints ensure that the optimization respects the stochastic nature of
the inventory dynamics and the stopping criterion at $\tau $.

\subsection{Hamilton-Jacobi-Bellman Equations}

To solve the optimization problem, we employ the value function approach.
The value function is defined as:

\begin{equation}
z_{i}(x)=\inf_{p_{1},\dots ,p_{N}}\mathbb{E}\left[ \int_{0}^{\tau }\left(
\sum_{j=1}^{N}p_{j}^{2}(t,\epsilon (t))+f_{\epsilon (t)}(y(t,\epsilon
(t)))\right) e^{-\alpha _{\epsilon (t)}t}dt\Bigg|y(0)=x,\epsilon (0)=i\right]
.  \label{hjb}
\end{equation}

The HJB equations for the value functions $z_{1}(x)$ and $z_{2}(x)$,
corresponding to the two regimes $\epsilon (t)=1$ and $\epsilon (t)=2$, are
given by:%
\begin{equation}
\left\{ 
\begin{array}{ccc}
-a_{1}z_{2}+(a_{1}+\alpha _{1})z_{1}-\frac{{\sigma _{1}}^{2}}{2}\Delta
z_{1}-f_{1}\left( x\right) =-\frac{1}{4}\left\vert \nabla z_{1}\right\vert
^{2}, & \text{for} & x\in B_{R}, \\ 
-a_{2}z_{1}+(a_{2}+\alpha _{2})z_{2}-\frac{{\sigma _{2}}^{2}}{2}\Delta
z_{2}-f_{2}\left( x\right) =-\frac{1}{4}\left\vert \nabla z_{2}\right\vert
^{2}, & \text{for} & x\in B_{R},%
\end{array}%
\right.  \label{sp}
\end{equation}%
with boundary conditions:

\begin{equation*}
z_{1}(x)=z_{2}(x)=0,\quad \text{for }x\in \partial B_{R}.
\end{equation*}%
Here, $\ a_{1},a_{2},\alpha _{1},\alpha _{2},\sigma _{1},\sigma _{2}$ are
regime-dependent parameters, $\Delta z_{i}$ is the Laplacian of $z_{i}(x)$
(sum of second-order partial derivatives), $B_{R}$ is the open ball in $%
\mathbb{R}^{N}$ ($N\geq 1$) of radius $R>0$, and $f_{1}(x)$, $f_{2}(x)$ are
the holding cost functions in regimes 1 and 2, respectively.

\subsubsection{Assumptions}

To ensure mathematical tractability, we impose the following assumptions:

\begin{itemize}
\item \quad $f_{1}(x),$ $f_{2}(x)$ are continuous, convex functions
satisfying $f_{i}(x)\leq M_{i}x^{2}$, $i=1,2$;

\item \quad $\sigma _{\epsilon (t)}>0$ and $\alpha _{\epsilon (t)}>0$,
ensuring non-degenerate stochastic dynamics;

\item \quad boundary conditions: $z_{\epsilon (t)}=0$ when $t=\tau $, $%
i=1,\dots ,N$.
\end{itemize}

This formulation provides the mathematical foundation for deriving the
Hamilton-Jacobi-Bellman equations and solving the optimization problem.

The computational goal is to approximate the value functions $z_{1}(x)$ and $%
z_{2}(x)$ using numerical techniques that guarantee convergence and
stability.

The next section focuses on the methodology used to obtain the solutions.

\subsection{Transformation and Simplification}

To simplify the PDE system (\ref{sp}), we apply a change of variables:

\begin{equation*}
z_{j}(x)=-2\sigma _{j}^{2}\ln u_{j}(x),\text{ }j=1,2,
\end{equation*}%
which removes the gradient terms and transforms the PDE system into:%
\begin{equation}
\left\{ 
\begin{array}{ccc}
\Delta u_{1}\left( x\right) =u_{1}\left( x\right) \left[ \frac{1}{{\sigma
_{1}^{4}}}f_{1}\left( x\right) +\frac{2(a_{1}+\alpha _{1})}{{\sigma _{1}^{2}}%
}\ln u_{1}\left( x\right) -2a_{1}\frac{{\sigma _{2}^{2}}}{{\sigma _{1}^{4}}}%
\ln u_{2}\left( x\right) \right] , & for & x\in B_{R} \\ 
\Delta u_{2}\left( x\right) =u_{2}\left( x\right) \left[ \frac{1}{{\sigma }%
_{2}^{4}}f_{2}\left( x\right) +\frac{2(a_{2}+\alpha _{2})}{{\sigma _{2}^{2}}}%
\ln u_{2}\left( x\right) -2a_{2}\frac{{\sigma _{1}^{2}}}{{\sigma _{2}^{4}}}%
\ln u_{1}\left( x\right) \right] , & for & x\in B_{R}%
\end{array}%
\right.  \label{sl}
\end{equation}%
with boundary conditions:

\begin{equation*}
u_{1}(x)=u_{2}(x)=1,\quad x\in \partial B_{R}.
\end{equation*}%
This transformation reduces the complexity of the system and facilitates
numerical computation.

\subsection{Existence and Uniqueness of Solutions}

The solution's computation involved specific parameters that had to be
determined in an approximately exact form. In the paper \cite{CCP}, we
established only the existence of these parameters. Therefore, it becomes
essential to provide a proof of the results that will facilitate our
computational technique. Consequently, to implement our main findings, we
demonstrate the following practical result:

\begin{lemma}
\label{l1}For any $a_{1},\alpha _{1},a_{2},\alpha _{2},\sigma _{1},\sigma
_{2},M_{1},M_{2},R\in (0,\infty )$ and $N\in \mathbb{N}^{\ast }$, there
exist and are unique $K_{1},K_{2}\in (-\infty ,0)$ such that: 
\begin{equation}
\left\{ 
\begin{array}{l}
4K_{1}^{2}+\frac{2(a_{1}+\alpha _{1}){\sigma _{1}^{2}}}{{\sigma _{1}^{4}}}%
K_{1}-\frac{M_{1}}{{\sigma _{1}^{4}}}-\frac{2a_{1}{\sigma _{2}^{2}}}{{\sigma
_{1}^{4}}}K_{2}=0, \\ 
4K_{2}^{2}+\frac{2(a_{2}+\alpha _{2}){\sigma _{2}^{2}}}{{\sigma _{2}^{4}}}%
K_{2}-\frac{M_{2}}{{\sigma _{2}^{4}}}-\frac{2a_{2}{\sigma _{1}^{2}}}{{\sigma
_{2}^{4}}}K_{1}=0, \\ 
-\frac{2(a_{1}+\alpha _{1})R^{2}}{{\sigma _{1}^{2}}}K_{1}-2K_{1}N+\frac{%
2a_{1}{\sigma _{2}^{2}}R^{2}}{{\sigma _{1}^{4}}}K_{2}\geq 0, \\ 
-\frac{2(a_{2}+\alpha _{2})R^{2}}{{\sigma _{2}^{2}}}K_{2}-2NK_{2}+\frac{%
2a_{2}{\sigma _{1}^{2}}R^{2}}{{\sigma _{2}^{4}}}K_{1}\geq 0.%
\end{array}%
\right.  \label{non}
\end{equation}
\end{lemma}

\begin{proof}
The inequalities are quadratic in nature with respect to $K_{1}$ and $K_{2}$%
. We analyze the first equation:%
\begin{equation*}
4K_{1}^{2}+\frac{2(a_{1}+\alpha _{1})\sigma _{1}^{2}}{\sigma _{1}^{4}}K_{1}-%
\frac{M_{1}}{\sigma _{1}^{4}}-\frac{2a_{1}\sigma _{2}^{2}}{\sigma _{1}^{4}}%
K_{2}=0.
\end{equation*}%
This can be rewritten in standard quadratic form:%
\begin{equation*}
4K_{1}^{2}+\frac{2(a_{1}+\alpha _{1})}{\sigma _{1}^{2}}K_{1}+\left( -\frac{%
M_{1}}{\sigma _{1}^{4}}-\frac{2a_{1}\sigma _{2}^{2}}{\sigma _{1}^{4}}%
K_{2}\right) =0.
\end{equation*}%
The discriminant of this quadratic equation is non-negative for $K_{1}$:%
\begin{equation*}
\Delta _{1}=\left( \frac{2(a_{1}+\alpha _{1})}{\sigma _{1}^{2}}\right)
^{2}-4\cdot 4\cdot \left( -\frac{M_{1}}{\sigma _{1}^{4}}-\frac{2a_{1}\sigma
_{2}^{2}}{\sigma _{1}^{4}}K_{2}\right) \geq 0,
\end{equation*}%
and so the equation have real solutions. A similar process applies to the
second equation:%
\begin{equation*}
\Delta _{2}=\left( \frac{2(a_{2}+\alpha _{2})}{\sigma _{2}^{2}}\right)
^{2}-4\cdot 4\cdot \left( -\frac{M_{2}}{\sigma _{2}^{4}}-\frac{2a_{2}\sigma
_{1}^{2}}{\sigma _{2}^{4}}K_{1}\right) \geq 0.
\end{equation*}%
Let:%
\begin{equation*}
K_{1}^{\ast }=\frac{-\frac{2(a_{1}+\alpha _{1})}{\sigma _{1}^{2}}-\sqrt{%
\Delta _{1}}}{8}\in \left( -\infty ,0\right) ,\quad K_{2}^{\ast }=\frac{-%
\frac{2(a_{2}+\alpha _{2})}{\sigma _{2}^{2}}-\sqrt{\Delta _{2}}}{8}\in
\left( -\infty ,0\right) .
\end{equation*}%
Define:%
\begin{equation*}
R_{1}(K_{1})=4K_{1}^{2}+\frac{2(a_{1}+\alpha _{1})\sigma _{1}^{2}}{\sigma
_{1}^{4}}K_{1}-\frac{M_{1}}{\sigma _{1}^{4}}-\frac{2a_{1}\sigma _{2}^{2}}{%
\sigma _{1}^{4}}K_{2}^{\ast },
\end{equation*}%
and%
\begin{equation*}
R_{2}(K_{2})=4K_{2}^{2}+\frac{2(a_{2}+\alpha _{2})\sigma _{2}^{2}}{\sigma
_{2}^{4}}K_{2}-\frac{M_{2}}{\sigma _{2}^{4}}-\frac{2a_{2}\sigma _{1}^{2}}{%
\sigma _{2}^{4}}K_{1}^{\ast }.
\end{equation*}%
Observe:%
\begin{equation*}
R_{1}(0)=-\frac{M_{1}}{\sigma _{1}^{4}}-\frac{2a_{1}\sigma _{2}^{2}}{\sigma
_{1}^{4}}K_{2}^{\ast }<0,\quad R_{2}(0)=-\frac{M_{2}}{\sigma _{2}^{4}}-\frac{%
2a_{2}\sigma _{1}^{2}}{\sigma _{2}^{4}}K_{1}^{\ast }<0.
\end{equation*}%
On the other hand:%
\begin{equation*}
\lim_{K_{1}\rightarrow -\infty }R_{1}(K_{1})=+\infty ,\quad
\lim_{K_{2}\rightarrow -\infty }R_{2}(K_{2})=+\infty .
\end{equation*}%
Thus, there exist $K_{1},K_{2}\in (-\infty ,0)$ such that the first and
second equations of the system are satisfied. By the monotonicity of $R_{1}$
and $R_{2}$, these solutions $K_{1}$ and $K_{2}$ are unique. It is evident
that the third and fourth inequalities of (\ref{non}) hold true for any $%
K_{1},K_{2}\in (-\infty ,0)$, and in particular, they are satisfied for the
specifically chosen parameters.
\end{proof}

We are now ready to adapt the proof of the Theorem in \cite{CCP},
integrating the necessary steps to address the numerical implications.

\begin{theorem}
\label{main}The system of equations (\ref{sp}) has a unique positive
solution $(z_{1},z_{2})\in \lbrack C^{2}(B_{R})\cap C(\overline{B}_{R})]^{2}$
with value functions $z_{1}$ and $z_{2}$ such that 
\begin{equation*}
z_{i}\left( x\right) \leq B_{i}\left( x\right) ,\text{ where }B_{1}\left(
x\right) =-2\sigma _{1}^{2}K_{1}(R^{2}-|x|^{2})\text{, }B_{2}\left( x\right)
=-2\sigma _{2}^{2}K_{2}(R^{2}-|x|^{2})
\end{equation*}%
and $K_{1},K_{2}\in \left( -\infty ,0\right) $ are the unique solutions of
the nonlinear system (\ref{non}).
\end{theorem}

\begin{proof}
Our constructive approach aims to develop a computational scheme for
numerical approximations of the solution. Since the system (\ref{sp}) is
equivalent to (\ref{sl}), we will focus on the latter. The approach involves
three key steps:

\textbf{Step 1:} Sub-solution and Super-solution Construction.

The main problem reduces to constructing functions $(u_{1},u_{2})$ as
sub-solutions (and $(u_{1},u_{2})$ as super-solutions) for the system (\ref%
{sl}) that satisfy the inequalities:%
\begin{equation*}
\left\{ 
\begin{array}{lll}
\Delta u_{1}(x)\geq u_{1}(x)\left[ \frac{1}{\sigma _{1}^{4}}f_{1}(x)+\frac{%
2(a_{1}+\alpha _{1})}{\sigma _{1}^{2}}\ln u_{1}(x)-\frac{2a_{1}\sigma
_{2}^{2}}{\sigma _{1}^{4}}\ln u_{2}(x)\right] , & for & x\in B_{R}, \\ 
\Delta u_{2}(x)\geq u_{2}(x)\left[ \frac{1}{\sigma _{2}^{4}}f_{2}(x)+\frac{%
2(a_{2}+\alpha _{2})}{\sigma _{2}^{2}}\ln u_{2}(x)-\frac{2a_{2}\sigma
_{1}^{2}}{\sigma _{2}^{4}}\ln u_{1}(x)\right] , & for & x\in B_{R}, \\ 
(u_{1}(x),u_{2}(x))=\left( 1,1\right) , & for & x\in \partial B_{R},%
\end{array}%
\right.
\end{equation*}%
(and similarly for the inequalities with $\leq $). For sub-solutions, choose:%
\begin{equation*}
(\underline{u}_{1}(x),\underline{u}_{2}(x))=\left(
e^{K_{1}(R^{2}-|x|^{2})},e^{K_{2}(R^{2}-|x|^{2})}\right) ,
\end{equation*}%
where $K_{1},K_{2}\in (-\infty ,0)$ are solutions of (\ref{non}). For
super-solutions, choose:%
\begin{equation*}
(\overline{u}_{1}(x),\overline{u}_{2}(x))=(1,1).
\end{equation*}%
Clearly,%
\begin{equation*}
\underline{u}_{1}(x)\leq \overline{u}_{1}(x)\text{ and }\underline{u}%
_{2}(x)\leq \overline{u}_{2}(x)\text{, for any }x\in \overline{B}_{R}.\text{ 
}
\end{equation*}%
\textbf{Step 2:} Approximation Scheme.

Construct sequences $\{(u_{1}^{k},u_{2}^{k})\}_{k\in \mathbb{N}}$ via
monotone Picard iterations starting with:%
\begin{equation*}
(u_{1}^{0},u_{2}^{0})=(\underline{u}_{1},\underline{u}_{2}),\quad x\in 
\overline{B}_{R}.
\end{equation*}%
Define the iteration:%
\begin{equation*}
\left\{ 
\begin{array}{ccc}
\Delta u_{1}^{k}+\Lambda
_{1}u_{1}^{k}=g_{1}(x,u_{1}^{k-1},u_{2}^{k-1})+\Lambda _{1}u_{1}^{k-1}, & for
& x\in B_{R}, \\ 
\Delta u_{2}^{k}+\Lambda
_{2}u_{2}^{k}=g_{2}(x,u_{1}^{k-1},u_{2}^{k-1})+\Lambda _{2}u_{2}^{k-1}, & for
& x\in B_{R},%
\end{array}%
\right.
\end{equation*}%
where for $S_{1}=e^{K_{1}R^{2}}$, $S_{2}=e^{K_{2}R^{2}}$, and $S=1$ the
functions:%
\begin{equation*}
g_{1}:\overline{B}_{R}\times \lbrack S_{1},S]\times \lbrack
S_{2},S]\rightarrow \mathbb{R}\text{ and }g_{2}:\overline{B}_{R}\times
\lbrack S_{2},S]\times \lbrack S_{1},S]\rightarrow \mathbb{R}
\end{equation*}%
are defined by: 
\begin{eqnarray}
g_{1}\left( x,t,s\right) &=&\frac{1}{\sigma _{1}^{4}}f_{1}\left( x\right) t+%
\frac{2(a_{1}+\alpha _{1})}{{\sigma _{1}^{2}}}t\ln t-\frac{2a_{1}{\sigma
_{2}^{2}}}{{\sigma _{1}^{4}}}t\ln s,\text{ }  \notag \\
&&  \label{26a} \\
g_{2}\left( x,t,s\right) &=&\frac{1}{{\sigma }_{2}^{4}}f_{2}\left( x\right)
s+\frac{2(a_{2}+\alpha _{2})}{{\sigma _{2}^{2}}}s\ln s-\frac{2a_{2}{\sigma
_{1}^{2}}}{{\sigma _{2}^{4}}}s\ln t.  \notag
\end{eqnarray}%
Since $g_{1}$ (respectively $g_{2}$) is a continuous function with respect
to the first variable in $B_{R}$ and continuously differentiable with
respect to the second and third in $[S_{1},S]\times \lbrack S_{2},S]$
(respectively, $[S_{2},S]\times \lbrack S_{1},S]$), it allows choosing $%
\Lambda _{1},\Lambda _{2}\in (-\infty ,0)$ such that: 
\begin{equation}
-\Lambda _{1}\geq \frac{g_{1}\left( x,t_{1},s\right) -g_{1}\left(
x,t_{2},s\right) }{t_{2}-t_{1}}\text{ (respectively}-\Lambda _{2}\geq \frac{%
g_{2}\left( x,t,s_{1}\right) -g_{2}\left( x,t,s_{2}\right) }{s_{2}-s_{1}}%
\text{)}  \label{27a}
\end{equation}%
for every $t_{1},t_{2}$ with $\underline{u}_{1}\leq t_{2}<t_{1}\leq 
\overline{u}_{1}$, $\underline{u}_{2}\leq s\leq \overline{u}_{2}$ and $x\in 
\overline{B}_{R}$ (respectively for every $s_{1},s_{2}$ with $\underline{u}%
_{2}\leq s_{2}<s_{1}\leq \overline{u}_{2}$, $\underline{u}_{1}\leq t\leq 
\overline{u}_{1}$ and $x\in \overline{B}_{R})$, to ensure monotonicity:%
\begin{equation*}
(u_{1}^{k-1}\leq u_{1}^{k},\,u_{2}^{k-1}\leq u_{2}^{k})\implies
(u_{1}^{k}\leq u_{1}^{k+1},\,u_{2}^{k}\leq u_{2}^{k+1}),
\end{equation*}%
via mathematical induction and the maximum principle.

\textbf{Step 3: }Convergence and Uniqueness.

The sequences $\{u_{1}^{k},u_{2}^{k}\}_{k\in \mathbb{N}}$ converge
monotonically to bounded limits:%
\begin{equation*}
\lim_{k\rightarrow \infty
}(u_{1}^{k}(x),u_{2}^{k}(x))=(u_{1}(x),u_{2}(x)),\quad x\in \overline{B}_{R}.
\end{equation*}%
Standard bootstrap arguments ensure:%
\begin{equation*}
(u_{1}^{k},u_{2}^{k})\rightarrow (u_{1},u_{2}),\quad \text{in }%
[C^{2}(B_{R})\cap C(\overline{B}_{R})]^{2},
\end{equation*}%
and $(u_{1},u_{2})$ solves (\ref{sl}) with:%
\begin{equation*}
\underline{u}_{1}(x)\leq u_{1}(x)\leq \overline{u}_{1}(x),\quad \underline{u}%
_{2}(x)\leq u_{2}(x)\leq \overline{u}_{2}(x),\quad x\in \overline{B}_{R}.
\end{equation*}%
Uniqueness follows from the maximum principle: any two positive solutions $%
(u_{1},u_{2})$ and $(\widetilde{u}_{1},\widetilde{u}_{2})$ coincide.
\end{proof}

\section{Optimal Control\label{3}}

The optimal control represents the production rate policy that minimizes the
total expected cost functional. It is derived using the
Hamilton-Jacobi-Bellman (HJB) equations and is directly related to the
gradients of the value functions. Below, we delve deeper into its
formulation and derivation.

\subsection{Optimal Production Policy}

By differentiating the HJB equations with respect to the inventory levels $%
y_{i}$, we obtain the gradient terms that define the optimal control.

Hence, for each economic good $i=1,\dots ,N$, the optimal production rate $%
p_{i}^{\ast }(t,\epsilon (t))$ is given by:

\begin{equation*}
p_{i}^{\ast }(t,\epsilon (t))=-\frac{1}{2}\frac{\partial z_{\epsilon (t)}}{%
\partial y_{i}},\text{ }
\end{equation*}%
where: $z_{\epsilon (t)}$ is the value function corresponding to regime $%
\epsilon (t)$ and $\frac{\partial z_{\epsilon (t)}}{\partial y_{i}}$ denotes
the partial derivative of the value function with respect to the inventory
level of good $i$.

This result is obtained by solving the first-order optimality condition
derived from the HJB equations.

\subsubsection*{Economic Interpretation}

The negative gradient of the value function implies that the optimal
production rate decreases as the marginal cost of inventory increases.
Intuitively: higher inventory levels (positive gradient) lead to a reduction
in production to avoid excess costs and lower inventory levels (negative
gradient) necessitate an increase in production to meet anticipated demand.

\subsection{Verification of Optimality}

The verification of optimality establishes that the control $p^{\ast
}(t,\epsilon (t))$, derived from the Hamilton-Jacobi-Bellman (HJB)
equations, is indeed the optimal control that minimizes the cost functional.
This involves proving the supermartingale property of the value function for
all admissible controls and the martingale property for the optimal control.

\subsubsection*{The stochastic process}

To verify that $p_{i}^{\ast }(t,\epsilon (t))$ is indeed the optimal
control, we use the supermartingale and martingale properties of the value
function $z_{\epsilon (t)}(y)$. Let the stochastic process $Z_{p}(t)$ be
defined as:

\begin{equation}
Z_{p}(t)=-e^{-\alpha _{\epsilon (t)}t}z_{\epsilon (t)}(y(t,\epsilon
(t)))-\int_{0}^{t}\left( \sum_{i=1}^{N}p_{i}^{2}(s,\epsilon (s))+f_{\epsilon
(s)}(y(s,\epsilon (s)))\right) e^{-\alpha _{\epsilon (s)}s}ds.  \label{spa}
\end{equation}%
where: $z_{\epsilon (t)}(y)$ is the value function for regime $\epsilon (t)$%
, $p_{i}(t,\epsilon (t))$ are the production rates (control variables), $%
f_{\epsilon (t)}(y)$ is the holding cost function, and $\alpha _{\epsilon
(t)}$ is the regime-dependent discount rate.

Using It\^{o}'s Lemma, the time derivative of $Z_{p}(t)$ satisfies:

\begin{equation*}
dZ_{p}(t)=e^{-\alpha _{\epsilon (t)}t}\left[ \alpha _{\epsilon
(t)}z_{\epsilon (t)}(y)+\mathcal{L}z_{\epsilon
(t)}(y)-\sum_{i=1}^{N}p_{i}^{2}(t,\epsilon (t))-f_{\epsilon (t)}(y)\right]
dt+M(t),
\end{equation*}%
where $M(t)$ is a martingale term, and $\mathcal{L}z_{\epsilon (t)}(y)$
represents the generator of the Markov-modulated diffusion.

\subsubsection*{Supermartingale and Martingale Properties}

For the admissible controls $p_{i}(t,\epsilon (t))$, $Z_{p}(t)$ is a
supermartingale because $\mathcal{L}z_{\epsilon (t)}(y)$ satisfies the HJB
inequality:

\begin{equation*}
\alpha _{\epsilon (t)}z_{\epsilon (t)}(y)+\mathcal{L}z_{\epsilon
(t)}(y)-\sum_{i=1}^{N}p_{i}^{2}(t,\epsilon (t))-f_{\epsilon (t)}(y)\leq 0.
\end{equation*}

For the optimal control $p^*(t, \epsilon(t)) $, $Z_p(t) $ is a martingale
because $\mathcal{L}z_{\epsilon(t)}(y) $ satisfies the equality condition:

\begin{equation*}
\alpha _{\epsilon (t)}z_{\epsilon (t)}(y)+\mathcal{L}z_{\epsilon
(t)}(y)-\sum_{i=1}^{N}p_{i}^{\ast 2}(t,\epsilon (t))-f_{\epsilon (t)}(y)=0.
\end{equation*}

The optimal control $p^{\ast }(t,\epsilon (t))$ ensures that $Z_{p}(t)$ is a
martingale, while any other control $p(t,\epsilon (t))$ results in $Z_{p}(t)$
being a supermartingale. This validates the optimality of $p^{\ast
}(t,\epsilon (t))$.

\subsubsection*{Boundary Conditions and Optimality}

The boundary condition $z_{\epsilon(t)}(y) = 0 $ for $y \in \partial B_R $,
where $B_R $ is the ball of radius $R $, ensures that $Z_p(t) $ vanishes at
the stopping time $\tau $. Thus, for $t \geq \tau $, the contribution to the
cost functional ceases, confirming the proper termination of production when
inventory exceeds the threshold $R $.

\subsubsection*{Proof of Optimality}

The optimality of $p^{\ast }(t,\epsilon (t))$ is formalized through the
following theorem:

\begin{theorem}
\label{main2}Let $Z_{p}(t)$ be the stochastic process defined in (\ref{spa}%
). The control $p^{\ast }(t,\epsilon (t))$, derived from the HJB equations,
minimizes the cost functional $J(p_{1},\dots ,p_{N})$ and satisfies:%
\begin{equation*}
Z_{p^{\ast }}(t)=z_{\epsilon (t)}(y(0)).
\end{equation*}
\end{theorem}

\begin{proof}
1.\quad For $Z_{p}(t)$ under the admissible control $p(t,\epsilon (t))$:%
\begin{equation*}
\mathbb{E}[Z_{p}(\tau )]\leq Z_{p}(0),
\end{equation*}%
since $Z_{p}(t)$ is a supermartingale.

2.\quad For $Z_{p^{\ast }}(t)$ under the optimal control $p^{\ast
}(t,\epsilon (t))$:%
\begin{equation*}
\mathbb{E}[Z_{p^{\ast }}(\tau )]=Z_{p^{\ast }}(0),
\end{equation*}%
since $Z_{p^{\ast }}(t)$ is a martingale.

3.\quad From the boundary condition $z_{\epsilon (t)}(y)=0$, it follows that:%
\begin{equation*}
\mathbb{E}\left[ \int_{0}^{\tau }\left( \sum_{i=1}^{N}p_{i}^{2}(s,\epsilon
(s))+f_{\epsilon (s)}(y(s,\epsilon (s)))\right) e^{-\alpha _{\epsilon
(s)}s}ds\right] =z_{\epsilon (t)}(y(0)).
\end{equation*}%
Thus, the control $p^{\ast }(t,\epsilon (t))$ minimizes the cost functional $%
J(p_{1},\dots ,p_{N})$ and satisfies the optimality condition.
\end{proof}

\subsubsection*{Theoretical Properties of the Optimal Control}

Theoretical properties of the optimal control are: the optimal control $%
p_{i}^{\ast }(t,\epsilon (t))$ is Lipschitz continuous in $y$, ensuring
stability in production rates under small changes in inventory levels, the
control policy is adaptive, responding dynamically to regime changes
governed by the Markov chain $\epsilon (t)$ and the quadratic nature of the
cost functional guarantees uniqueness of the optimal control.

\section{Numerical Computation\label{4}}

In this section, we present the numerical methodology to approximate the
solution of the system of elliptic partial differential equations (PDEs) (%
\ref{sp}) derived from the Hamilton-Jacobi-Bellman (HJB) equations (\ref{hjb}%
). This involves constructing sub- and super-solutions, implementing a
monotone iterative scheme, and addressing the computational challenges
inherent in the regime-switching framework.

\subsection{Mathematical Formulation of the Algorithm}

\subsubsection{Step 1: Steps to Find the Parameters $K_{1}$ and $K_{2}$}

\subparagraph{Step a. Define the Equations}

The parameters $K_{1}$ and $K_{2}$ are obtained by solving the following
system of nonlinear equations. The system of equations to solve for $K_{1}$
and $K_{2}$ is given by:%
\begin{equation*}
\left\{ 
\begin{array}{c}
4K_{1}^{2}+\frac{2(a_{1}+\alpha _{1})\sigma _{1}^{2}}{\sigma _{1}^{4}}K_{1}-%
\frac{M_{1}}{\sigma _{1}^{4}}-\frac{2a_{1}\sigma _{2}^{2}}{\sigma _{1}^{4}}%
K_{2}=0, \\ 
4K_{2}^{2}+\frac{2(a_{2}+\alpha _{2})\sigma _{2}^{2}}{\sigma _{2}^{4}}K_{2}-%
\frac{M_{2}}{\sigma _{2}^{4}}-\frac{2a_{2}\sigma _{1}^{2}}{\sigma _{2}^{4}}%
K_{1}=0.%
\end{array}%
\right.
\end{equation*}
These equations are solved numerically using the \texttt{fsolve} function
from Python.

\subparagraph{Step b. Choose an Initial Guess}

An initial guess for $K_1$ and $K_2$ must be provided to start the numerical
solution. Typically:

\begin{equation*}
K_{1}=K_{1}^{0}\in \left( -\infty ,0\right) ,\quad K_{2}=K_{2}^{0}\in \left(
-\infty ,0\right) .
\end{equation*}

\subparagraph{Step c. Solve the System Numerically}

The system is solved numerically using methods such as, iterative techniques
for nonlinear equations like Newton-Raphson, if implemented manually.

The \texttt{fsolve} function in Python is a numerical algorithm used to find
roots of a system of nonlinear equations. From a mathematical standpoint, it
operates as follows.

Let%
\begin{equation*}
\begin{array}{c}
h_{1}(K_{1},K_{2})=4K_{1}^{2}+\frac{2(a_{1}+\alpha _{1})\sigma _{1}^{2}}{%
\sigma _{1}^{4}}K_{1}-\frac{M_{1}}{\sigma _{1}^{4}}-\frac{2a_{1}\sigma
_{2}^{2}}{\sigma _{1}^{4}}K_{2}, \\ 
h_{2}(K_{1},K_{2})=4K_{2}^{2}+\frac{2(a_{2}+\alpha _{2})\sigma _{2}^{2}}{%
\sigma _{2}^{4}}K_{2}-\frac{M_{2}}{\sigma _{2}^{4}}-\frac{2a_{2}\sigma
_{1}^{2}}{\sigma _{2}^{4}}K_{1}.%
\end{array}%
\end{equation*}%
Consider a system of $2$ equations with $2$ variables:

\begin{equation}
\mathbf{H}(\mathbf{x})=%
\begin{bmatrix}
h_{1}(K_{1},K_{2}) \\ 
h_{2}(K_{1},K_{2})%
\end{bmatrix}%
=\mathbf{0},  \label{sr}
\end{equation}%
where $\mathbf{H}(\mathbf{x})$ is a vector-valued function, and $\mathbf{x}%
=[K_{1},K_{2}]^{\top }$ is the vector of unknowns and $\mathbf{0=}%
[0,0]^{\top }$. The goal of \texttt{fsolve} is to find $\mathbf{x}$ such
that $\mathbf{H}(\mathbf{x})=\mathbf{0}$.

The \texttt{fsolve} function in Python relies on iterative methods,
primarily Newton-Raphson, to approximate the solution of (\ref{sr}). The
iteration formula is:

\begin{equation*}
\mathbf{x}^{(k+1)}=\mathbf{x}^{(k)}-\mathbf{J}^{-1}(\mathbf{x}^{(k)})\mathbf{%
H}(\mathbf{x}^{(k)}),
\end{equation*}%
where: $\mathbf{x}^{(k)}$ is the solution estimate at iteration $k$, $%
\mathbf{J}(\mathbf{x})$ is the Jacobian matrix, defined as:%
\begin{equation*}
\mathbf{J}(\mathbf{x})=%
\begin{bmatrix}
\frac{\partial h_{1}}{\partial K_{1}} & \frac{\partial h_{1}}{\partial K_{2}}
\\ 
\frac{\partial h_{2}}{\partial K_{1}} & \frac{\partial h_{2}}{\partial K_{2}}%
\end{bmatrix}%
,
\end{equation*}%
and $\mathbf{J}^{-1}(\mathbf{x})$ is the inverse of the Jacobian matrix.

The Newton-Raphson method proceeds as follows:

\subparagraph{Step a.\quad}

Properly select the initial guess $\mathbf{x}^{(0)}$ to guarantee
convergence.

\subparagraph{Step b.\quad}

Evaluate the Jacobian $\mathbf{J}(\mathbf{x}^{(k)})$ at the current
iteration point $\mathbf{x}^{(k)}$.

\subparagraph{Step c.\quad}

Compute the next approximation $\mathbf{x}^{(k+1)}$ using the iteration
formula.

\subparagraph{Step d.}

Check the convergence if $\Vert \mathbf{H}(\mathbf{x}^{(k+1)})\Vert
<\epsilon $, where $\epsilon $ is the convergence threshold.

\subparagraph{Step e.\quad}

Repeat until convergence is achieved.

\subparagraph{Step f: Compute Derived Quantities}

Once $K_1$ and $K_2$ are determined, use them to compute additional
quantities:

\begin{equation*}
S_1 = \exp(K_1 R^2), \quad S_2 = \exp(K_2 R^2),
\end{equation*}

where $R$ is the domain radius.

\subparagraph{Remark.}

\texttt{fsolve} returns: solution vector $\mathbf{x}$ that satisfies $%
\mathbf{H}(\mathbf{x})\approx \mathbf{0}$, information about the convergence
status, such as success or failure of the algorithm, the solution is given
as:

\begin{equation*}
K_{1},K_{2}=\text{fsolve}(\text{equations},\text{initial guess}).
\end{equation*}

\subparagraph{Remark.}

The initial guess $\mathbf{x}^{(0)}$ significantly affects the convergence
and success of the algorithm. The Jacobian matrix $\mathbf{J}(\mathbf{x})$
is non-singular at each iteration, so the method succeed. If Newton-Raphson
fails, \texttt{fsolve} may fall back on other methods, such as quasi-Newton
or trust-region algorithms.

\subsubsection{Step 2: Defining Functions $g_{1}$ and $g_{2}$}

The functions $g_{1}(x,t,s)$ and $g_{2}(x,t,s)$ are defined as: 
\begin{align*}
g_{1}(x,t,s)& =\frac{f_{1}\left( x\right) }{\sigma _{1}^{4}}t+\frac{%
2(a_{1}+\alpha _{1})}{\sigma _{1}^{2}}t\ln (t)-\frac{2a_{1}\sigma _{2}^{2}}{%
\sigma _{1}^{4}}t\ln (s), \\
g_{2}(x,t,s)& =\frac{f_{2}\left( x\right) }{\sigma _{2}^{4}}s+\frac{%
2(a_{2}+\alpha _{2})}{\sigma _{2}^{2}}s\ln (s)-\frac{2a_{2}\sigma _{1}^{2}}{%
\sigma _{2}^{4}}s\ln (t).
\end{align*}

\subsubsection{Step 3: Calculating Partial Derivatives of $g_{1}$ and $g_{2}$%
}

The partial derivatives of the functions $g_{1}$ and $g_{2}$ are computed
as: 
\begin{align*}
\frac{\partial g_{1}}{\partial t}& =\frac{f_{1}\left( x\right) }{\sigma
_{1}^{4}}+\frac{2(a_{1}+\alpha _{1})}{\sigma _{1}^{2}}(1+\ln (t))-\frac{%
2a_{1}\sigma _{2}^{2}}{\sigma _{1}^{4}}\ln (s), \\
\frac{\partial g_{2}}{\partial s}& =\frac{f_{2}\left( x\right) }{\sigma
_{2}^{4}}+\frac{2(a_{2}+\alpha _{2})}{\sigma _{2}^{2}}(1+\ln (s))-\frac{%
2a_{2}\sigma _{1}^{2}}{\sigma _{2}^{4}}\ln (t).
\end{align*}

These derivatives are used to compute the maximum values $\Lambda_1$ and $%
\Lambda_2$:

\begin{equation*}
\Lambda_1 = -\max\left(\frac{\partial g_1}{\partial t}\right), \quad
\Lambda_2 = -\max\left(\frac{\partial g_2}{\partial s}\right).
\end{equation*}

\subsubsection{Step 4: Procedure to Compute $\Lambda _{1}$ and $\Lambda _{2}$%
}

The computation of $\Lambda_1$ and $\Lambda_2$ involves evaluating the
partial derivatives of $g_1$ and $g_2$ and determining their maxima over
specified domains.

\subparagraph{Step a: Discretize the Domain}

Define a grid of points $x$ over the domain $[-R,R]$, where $R$ is the
radius of the domain and for each $x$, compute the ranges for $t$ and $s$:%
\begin{equation*}
t\in \left[ \exp \left( K_{1}\left( R^{2}-\left\vert x\right\vert
^{2}\right) \right) ,1\right] ,\quad s\in \left[ \exp \left( K_{2}\left(
R^{2}-\left\vert x\right\vert ^{2}\right) \right) ,1\right] .
\end{equation*}

\subparagraph{Step b: Iterate Over the Grid}

The logarithmic terms in $g_{1}$ and $g_{2}$ require careful handling.
Newton's method or fixed-point iterations are employed to linearize the
equations at each step.

For each combination of $x$, $t$, and $s$ in their respective ranges,
compute the partial derivatives:%
\begin{equation*}
\frac{\partial g_{1}}{\partial t}\quad \text{and}\quad \frac{\partial g_{2}}{%
\partial s}.
\end{equation*}

Update the maximum values iteratively:

\begin{equation*}
\Lambda _{1}=-\max \left( \frac{\partial g_{1}}{\partial t}\right) ,\quad
\Lambda _{2}=-\max \left( \frac{\partial g_{2}}{\partial s}\right) .
\end{equation*}

\subparagraph{Step c: Negate the Results}

Finally, the values $\Lambda_1$ and $\Lambda_2$ are negated to convert the
maximum gradient into the required form:

\begin{equation*}
\Lambda _{1}=-\max \left( \frac{\partial g_{1}}{\partial t}\right) ,\quad
\Lambda _{2}=-\max \left( \frac{\partial g_{2}}{\partial s}\right) .
\end{equation*}%
The computation of $\Lambda _{1}$ and $\Lambda _{2}$ is typically
implemented using nested loops: outer loop (iterate over all values of $x$),
middle loop (iterate over all values of $t$ for a given $x$) and inner loop
(iterate over all values of $s$ for a given $x$ and $t$).

Choosing appropriate values of $\Lambda _{1}$ and $\Lambda _{2}$ ensures
that the scheme remains stable and convergent.

\subsubsection{Step 5: Monotone Iteration Scheme, with Initial Iteration and
Convergence Criteria}

The iteration scheme is implemented using a finite difference method: the
domain $B_{R}$ is discretized into a uniform grid with spacing $h>0$,
second-order central differences are used to approximate the Laplacian
operator $\Delta u_{j}$ and boundary conditions are enforced explicitly at
the grid points on $\partial B_{R}$.

In one dimension, the second derivative $\frac{\partial ^{2}u}{\partial x^{2}%
}$ at a point can be approximated using the second-order central difference
method:

\begin{equation*}
\frac{\partial ^{2}u}{\partial x^{2}}\approx \frac{u(x+h)-2u(x)+u(x-h)}{h^{2}%
}.
\end{equation*}

In $N$ dimensions, the Laplacian at a grid point $(i_{1},i_{2},\dots ,i_{n})$
is approximated as:

\begin{equation*}
\Delta u_{i_{1},i_{2},\dots ,i_{n}}\approx \sum_{d=1}^{N}\frac{%
u_{i_{1},\dots ,i_{d+1},\dots ,i_{n}}-2u_{i_{1},\dots ,i_{d},\dots
,i_{n}}+u_{i_{1},\dots ,i_{d-1},\dots ,i_{n}}}{h^{2}},
\end{equation*}%
where $d$ is the dimension index, and $u_{i_{1},\dots }$ refers to
neighboring values in the $d$-th direction. Since, this formula can be
extends naturally to higher dimensions by applying it independently along
each spatial direction we consider in the next the case $N=1$ and $h=\Delta
x $.

\subparagraph{Step a: Successive Approximation Method}

The successive approximation method involves solving for $u_{1}(x)$ and $%
u_{2}(x)$ iteratively: 
\begin{align*}
u_{1}^{(k+1)}[i]& =\frac{g_{1}(x[i],u_{1}^{(k)}[i],u_{2}^{(k)}[i])+\Lambda
_{1}u_{1}^{(k)}[i]-\frac{u_{1}[i-1]+u_{1}[i+1]}{\Delta x^{2}}}{-\frac{2}{%
\Delta x^{2}}+\Lambda _{1}}, \\
u_{2}^{(k+1)}[i]& =\frac{g_{2}(x[i],u_{1}^{(k)}[i],u_{2}^{(k)}[i])+\Lambda
_{2}u_{2}^{(k)}[i]-\frac{u_{2}[i-1]+u_{2}[i+1]}{\Delta x^{2}}}{-\frac{2}{%
\Delta x^{2}}+\Lambda _{2}}.
\end{align*}%
To initialize the successive approximation process, we start with
sub-solutions $u_{1}^{(0)}(x)$ and $u_{2}^{(0)}(x)$, typically defined as:

\begin{equation*}
u_{1}^{(0)}(x)=\exp \left( K_{1}\left( R^{2}-\left\vert x\right\vert
^{2}\right) \right) \text{ and }u_{2}^{(0)}(x)=\exp \left( K_{1}\left(
R^{2}-\left\vert x\right\vert ^{2}\right) \right) \text{ for all }\left\vert
x\right\vert =R.
\end{equation*}%
For the first iteration ($k=1$), we use the following update rules:%
\begin{equation*}
\left\{ 
\begin{array}{c}
u_{1}^{(1)}[i]=\frac{g_{1}(x[i],u_{1}^{(0)}[i],u_{2}^{(0)}[i])+\Lambda
_{1}u_{1}^{(0)}[i]-\frac{u_{1}[i-1]+u_{1}[i+1]}{\Delta x^{2}}}{-\frac{2}{%
\Delta x^{2}}+\Lambda _{1}}, \\ 
\\ 
u_{2}^{(1)}[i]=\frac{g_{2}(x[i],u_{1}^{(0)}[i],u_{2}^{(0)}[i])+\Lambda
_{2}u_{2}^{(0)}[i]-\frac{u_{2}[i-1]+u_{2}[i+1]}{\Delta x^{2}}}{-\frac{2}{%
\Delta x^{2}}+\Lambda _{2}}.%
\end{array}%
\right.
\end{equation*}
The new values of $u_{1}^{(1)}(x)$ and $u_{2}^{(1)}(x)$ are then used to
compute the subsequent iterations.

\subparagraph{Step b: Checking Convergence Criteria}

The convergence criteria are based on the maximum absolute difference
between the current and previous iterations. Define the convergence
thresholds $\epsilon > 0$ as:

\begin{equation*}
\max |u_1^{(k)}(x) - u_1^{(k-1)}(x)| < \epsilon, \quad \max |u_2^{(k)}(x) -
u_2^{(k-1)}(x)| < \epsilon.
\end{equation*}

The algorithm continues iterating until both criteria are satisfied: 
\begin{equation*}
\text{if }\max |u_{1}^{(k)}(x)-u_{1}^{(k-1)}(x)|\geq \epsilon \text{ or }%
\max |u_{2}^{(k)}(x)-u_{2}^{(k-1)}(x)|\geq \epsilon ,\text{ then continue
iterating.}
\end{equation*}

When the convergence criteria are satisfied, the algorithm terminates, and
the solutions $u_1(x)$ and $u_2(x)$ are considered stable.

\subparagraph{Step c: Where Convergence is Checked}

The convergence criteria are checked after each full iteration over the
domain $x \in [-R, R]$. Specifically, at the end of each iteration $k$, the
algorithm computes:

\begin{equation*}
\Delta u_1^{(k)} = \max |u_1^{(k)}(x) - u_1^{(k-1)}(x)|, \quad \Delta
u_2^{(k)} = \max |u_2^{(k)}(x) - u_2^{(k-1)}(x)|.
\end{equation*}

If both $\Delta u_1^{(k)}$ and $\Delta u_2^{(k)}$ are below $\epsilon$, the
algorithm exits the loop.

\subsubsection{Step 6: Value Functions and Optimal Rates}

The value functions are computed as:

\begin{equation*}
z_{1}(x)=-2\sigma _{1}^{2}\ln (u_{1}\left( x\right) ),\quad
z_{2}(x)=-2\sigma _{2}^{2}\ln (u_{2}\left( x\right) ).
\end{equation*}

The optimal production rates are calculated as:

\begin{equation*}
p_1^*(x) = -\frac{1}{2} \frac{\partial z_1}{\partial x}, \quad p_2^*(x) = -%
\frac{1}{2} \frac{\partial z_2}{\partial x}.
\end{equation*}

\subsubsection{Step 7: Upper Bound for Value Functions}

The upper-bound functions are defined as:

\begin{equation*}
B_1(x) = -2\sigma_1^2 K_1 (R^2 - x^2), \quad B_2(x) = -2\sigma_2^2 K_2 (R^2
- x^2).
\end{equation*}

\subsubsection{Step 8: Output and Convergence Properties}

The iterative scheme guarantees monotonic convergence to the unique solution
of the transformed PDE system:

\begin{equation*}
u_{1}^{k}(x)\rightarrow u_{1}(x),\quad u_{2}^{k}(x)\rightarrow
u_{2}(x),\quad \text{as }k\rightarrow \infty ,
\end{equation*}%
with the original value functions recovered as:

\begin{equation*}
z_{j}(x)=-2\sigma _{j}^{2}\ln u_{j}(x),\text{ }j=1,2.
\end{equation*}

\subsubsection{ Step 9: Simulation of Inventory Dynamics}

The inventory process is modeled by the stochastic differential equation
(SDE)

\begin{equation*}
dy(t)=p^{\ast }(t,\varepsilon (t))\,dt+\sigma _{\varepsilon
(t)}\,dW(t),\quad y(0)=x_{0},
\end{equation*}%
where the economic regime $\varepsilon (t)$ takes values in $\{1,2\}$ with
transitions:

\begin{equation*}
\mathbb{P}(\varepsilon (t+dt)=2\mid \varepsilon (t)=1)=a_{1}\,dt,\quad 
\mathbb{P}(\varepsilon (t+dt)=1\mid \varepsilon (t)=2)=a_{2}\,dt.
\end{equation*}%
In the simulation, one uses a time step $\Delta t$ and performs the
following update:

\begin{equation*}
y(t+\Delta t)=y(t)+p^{\ast }(y(t),\varepsilon (t))\,\Delta t+\sigma
_{\varepsilon (t)}\,\sqrt{\Delta t}\,\xi ,
\end{equation*}%
where $\xi \sim \mathcal{N}(0,1)$ is an independent standard normal random
variable. The control $p^{\ast }(\cdot )$ is determined via interpolation
from the computed optimal rates:%
\begin{equation*}
p^{\ast }(y(t),\varepsilon (t))=\left\{ 
\begin{array}{ccc}
-\frac{1}{2}\frac{\partial z_{1}}{\partial y} & if & \varepsilon (t)=1, \\ 
-\frac{1}{2}\frac{\partial z_{2}}{\partial y} & if & \varepsilon (t)=2.%
\end{array}%
\right.
\end{equation*}%
The simulation by Euler-Maruyama is executed until the stopping time

\begin{equation*}
\tau =\inf \{t>0:|y(t)|\geq R\},
\end{equation*}%
at which point the production is halted.

Numerical results, including the approximated value functions $z_{1}(x)$ and 
$z_{2}(x)$, are visualized to provide insights into the effects of regime
switching and sensitivity to model parameters. This facilitates
interpretation and validation of the results. Finally, the results,
including the dynamic of inventory $z_{1}(x)$, $z_{2}(x)$, $B_{1}(x)$, $%
B_{2}(x)$, $p_{1}^{\ast }(x)$ and $p_{2}^{\ast }(x)$ are visualized using
plots.

\section{Sensitivity, Model Analysis and Visualization\label{5}}

The proof of the results in this section is detailed in reference \cite{CCP}%
, and thus the specifics are excluded here. The data is presented to
visualize it in accordance with the results, showcasing the strength of the
mathematical conclusions and numerical implementation.

\subsection{Sensitivity Analysis}

The sensitivity analysis shows the following impacts:

\begin{theorem}
\label{s1}If $\alpha _{1}=\alpha _{2}$ and $f_{1}\left( x\right)
=f_{2}\left( x\right) $ for all $x\in B_{R}$, then higher volatility ($%
\sigma _{1}>\sigma _{2}$) increases the value function: $z_{1}(x)\geq
z_{2}(x)$ for all $x\in \overline{B}_{R}$.
\end{theorem}

\begin{theorem}
\label{s2}If $\sigma _{1}=\sigma _{2}$ and $f_{1}\left( x\right)
=f_{2}\left( x\right) $ for all $x\in B_{R}$, then higher discount rates ($%
\alpha _{1}<\alpha _{2}$) decrease the value function: $z_{1}(x)\geq
z_{2}(x) $ for all $x\in \overline{B}_{R}$.
\end{theorem}

\begin{theorem}
\label{s3}If $\alpha _{1}=\alpha _{2}$ and $\sigma _{1}=\sigma _{2}$, then
higher holding costs ($f_{1}(x)>f_{2}(x)$) increase the value function: $%
z_{1}(x)\geq z_{2}(x)$ for all $x\in \overline{B}_{R}$.
\end{theorem}

\subsection{Model Comparisons}

For models with and without regime switching, we have:

\begin{theorem}
\label{s4}If $\sigma _{1}>\sigma _{2}$, $\alpha _{1}<\alpha _{2}$ and $%
f_{1}\left( x\right) >f_{2}\left( x\right) $ for all $x\in B_{R}$, then:%
\begin{equation*}
\bar{z}_{1}(x)\geq z_{1}(x)\geq z_{2}(x)\geq \underline{z}_{2}(x),\text{ for
all }x\in B_{R}\text{,}
\end{equation*}%
where $\bar{z}_{1}(x)$ and $\underline{z}_{2}(x)$ corresponds to the value
function of a model without regime switching.
\end{theorem}

\subsection{Visualization of the solution \label{vs}}

The data provided 
\begin{equation*}
a_{1}=0.6\text{, }a_{2}=0.5\text{, }\alpha _{1}=\alpha _{2}=0.3\text{, }%
M_{1}=M_{2}=1\text{, }\sigma _{1}=1\text{, }\sigma _{2}=0.7
\end{equation*}%
and%
\begin{equation*}
f_{1}\left( x\right) =f_{2}\left( x\right) =x^{2},\text{ }x\in B_{20},
\end{equation*}%
adheres to Theorem \ref{s1}, ensuring that the plot shown in 
\begin{equation*}
\end{equation*}%

\begin{figure}[!ht]
\centering
\subfloat{\includegraphics[width=0.45\textwidth]{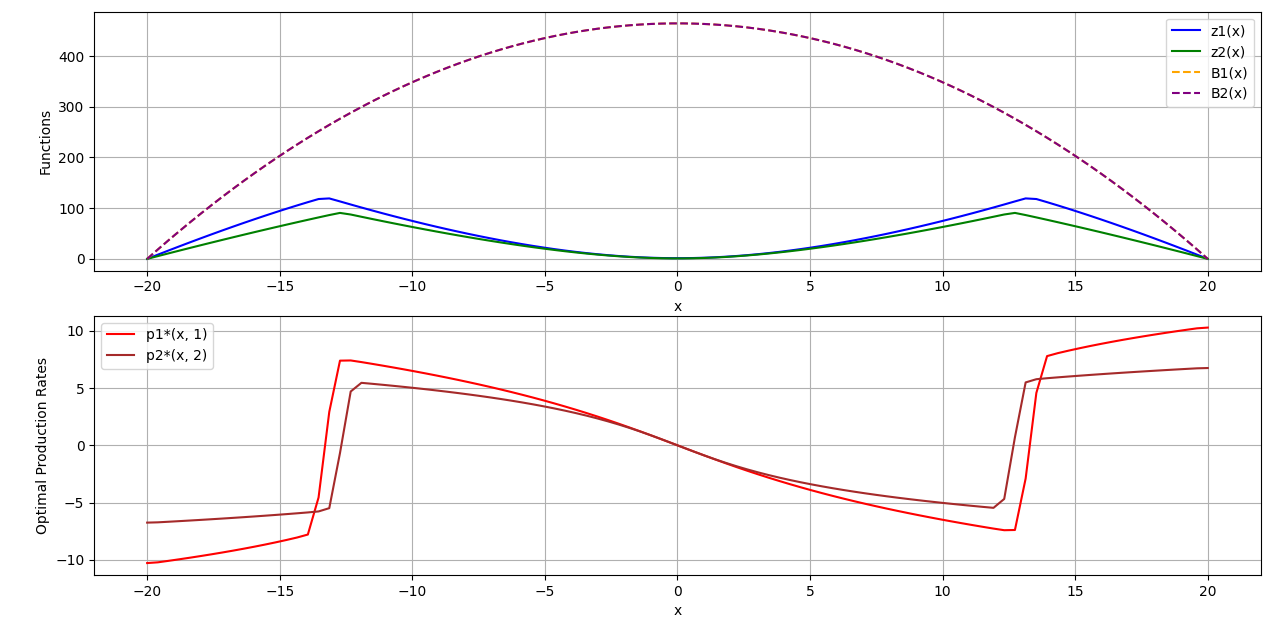}} %
\subfloat{\includegraphics[width=0.45\textwidth]{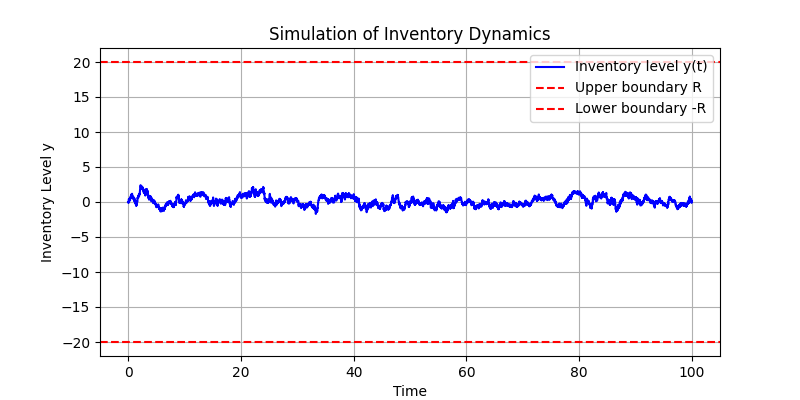}}
\caption{Caption describing the two figures.}
\label{fig:twoFigures1}
\end{figure}

\begin{equation*}
\end{equation*}%
aligns precisely with the derived results.

To comply with Theorem \ref{s2}, we utilize the specified data 
\begin{equation*}
a_{1}=0.6\text{, }a_{2}=0.5\text{, }\alpha _{1}=0.3\text{, }\alpha _{2}=0.7%
\text{, }M_{1}=M_{2}=1\text{, }\sigma _{1}=\sigma _{2}=1\text{, }
\end{equation*}%
and%
\begin{equation*}
f_{1}\left( x\right) =f_{2}\left( x\right) =x^{2},\text{ }x\in B_{20},
\end{equation*}%
leading to plots in 
\begin{equation*}
\end{equation*}

\begin{figure}[!ht]
\centering
\subfloat{\includegraphics[width=0.45\textwidth]{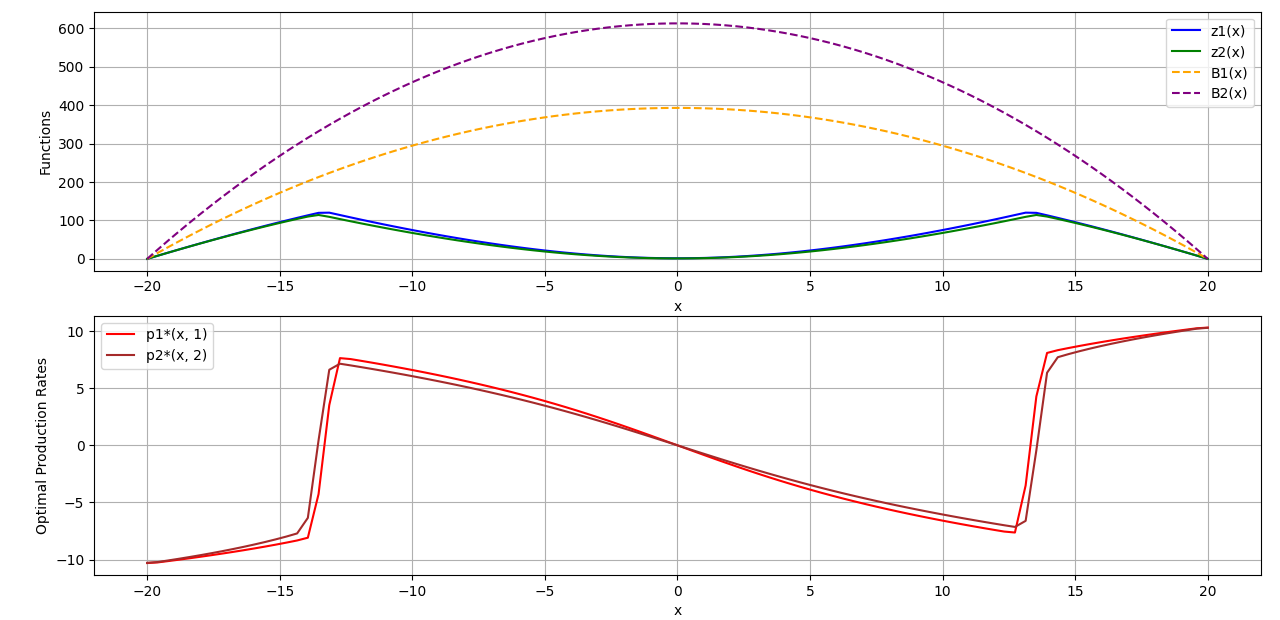}} %
\subfloat{\includegraphics[width=0.45\textwidth]{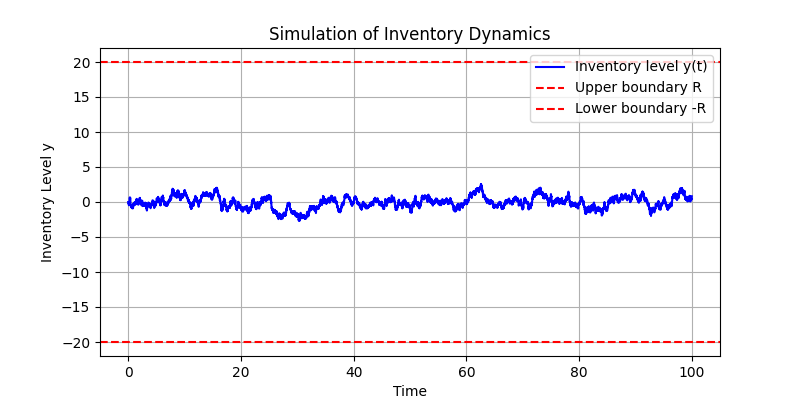}}
\caption{Caption describing the two figures.}
\label{fig:twoFigures2}
\end{figure}

\begin{equation*}
\end{equation*}
that are in exact agreement with the obtained findings.

By setting the data 
\begin{equation*}
a_{1}=0.6\text{, }a_{2}=0.9\text{, }\alpha _{1}=\alpha _{2}=0.3\text{, }%
M_{1}=5\text{, }M_{2}=1\text{, }\sigma _{1}=\sigma _{2}=1\text{,}
\end{equation*}%
and%
\begin{equation*}
f_{1}\left( x\right) =5x^{2}\text{, }f_{2}\left( x\right) =x^{2},\text{ }%
x\in B_{20},
\end{equation*}%
in accordance with Theorem \ref{s3}, the resulting plot in%
\begin{equation*}
\end{equation*}

\begin{figure}[!ht]
\centering
\subfloat{\includegraphics[width=0.45\textwidth]{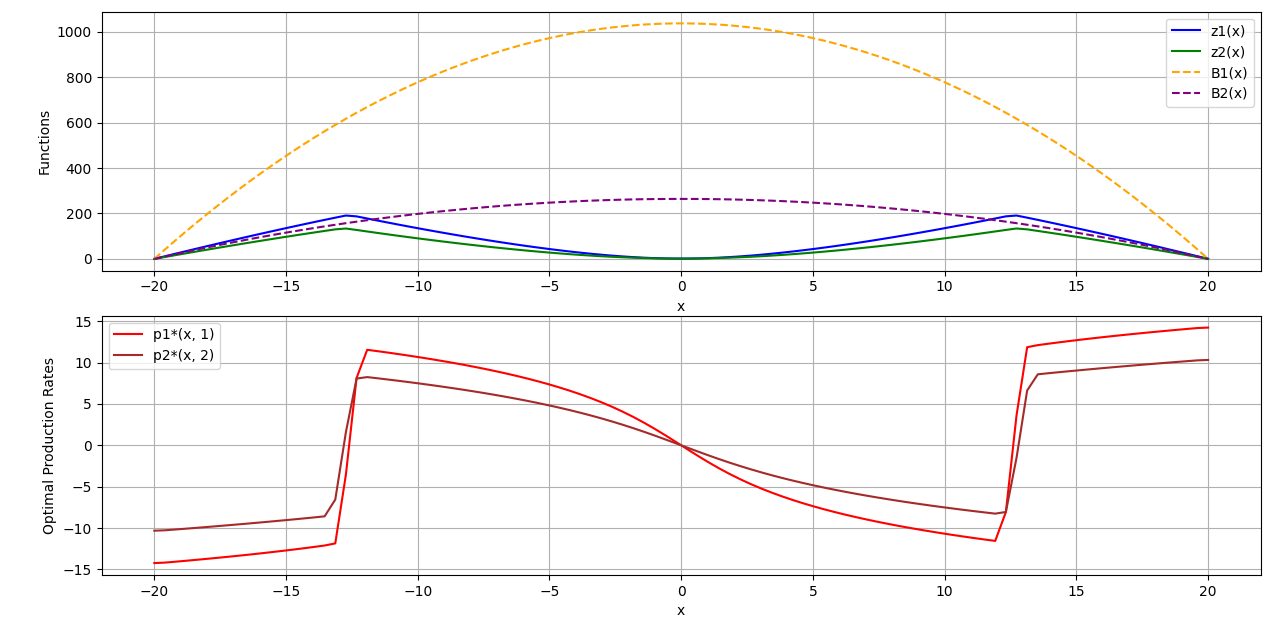}} %
\subfloat{\includegraphics[width=0.45\textwidth]{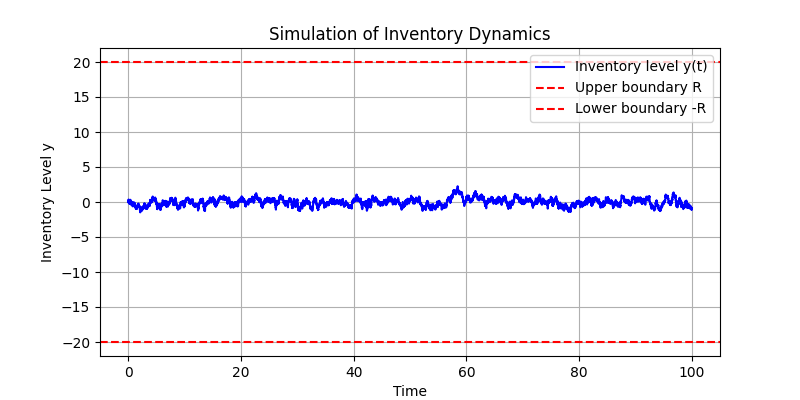}}
\caption{Caption describing the two figures.}
\label{fig:twoFigures3}
\end{figure}

\begin{equation*}
\end{equation*}
perfectly reflects the established outcomes.

The provided data 
\begin{equation*}
a_{1}=0.6\text{, }a_{2}=0.9\text{, }\alpha _{1}=0.3\text{, }\alpha _{2}=0.8%
\text{, }M_{1}=5\text{, }M_{2}=1\text{, }\sigma _{1}=1\text{, }\sigma
_{2}=0.3
\end{equation*}%
and%
\begin{equation*}
f_{1}\left( x\right) =5x^{2}\text{, }f_{2}\left( x\right) =x^{2},\text{ }%
x\in B_{10},
\end{equation*}%
in accordance with Theorem \ref{s4}, guarantees that the displayed plot%
\begin{equation*}
\end{equation*}

\begin{figure}[!ht]
\centering
\subfloat{\includegraphics[width=0.45\textwidth]{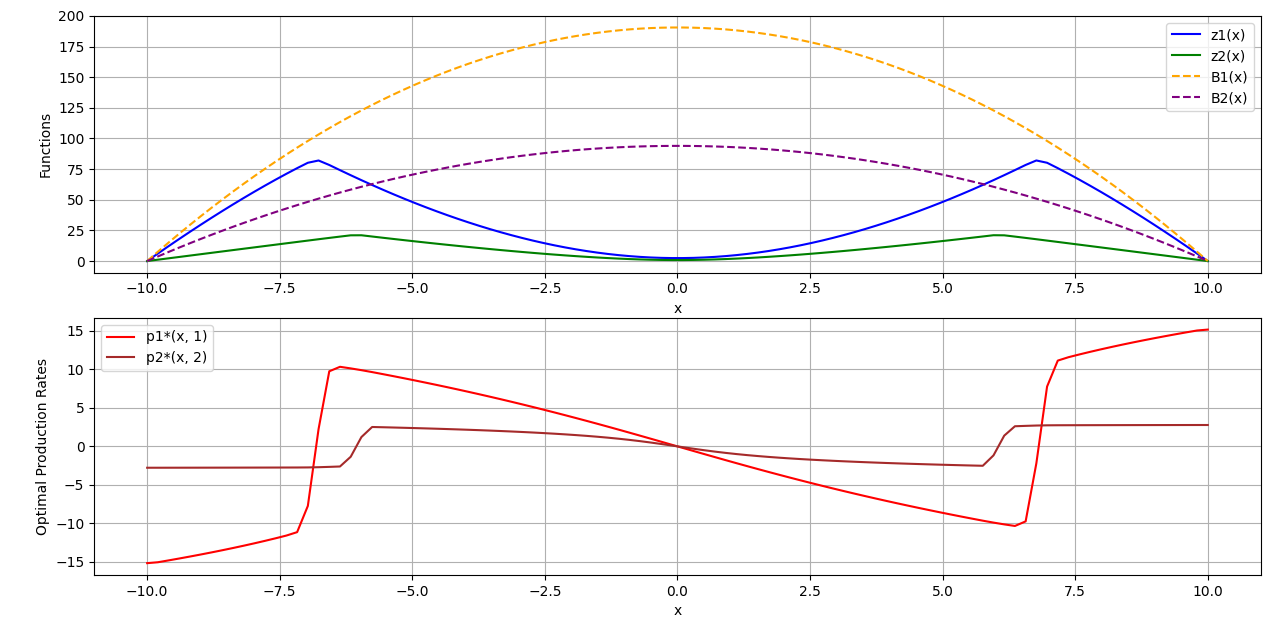}} %
\subfloat{\includegraphics[width=0.45\textwidth]{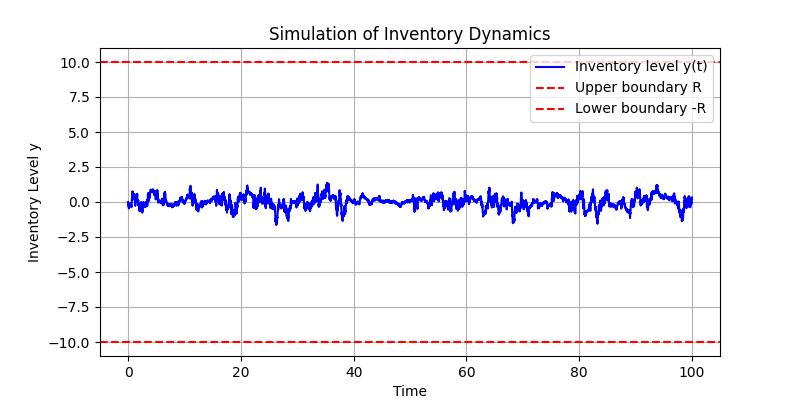}}
\caption{Caption describing the two figures.}
\label{fig:twoFigures4}
\end{figure}

\begin{equation*}
\end{equation*}
aligns seamlessly with the derived outcomes. It is unnecessary to plot the
functions $\bar{z}_{1}(x)$ and $\underline{z}_{2}(x)$ here, as the same
numerical algorithm is being repeated.

\begin{remark}
It is important to highlight that specific parameters require algorithmic
application, each with an associated margin of error. In all the above
considered scenarios, our theoretical results establish that: 
\begin{equation}
z_{1}\left( x\right) \leq B_{1}\left( x\right) \text{ and }z_{2}\left(
x\right) \leq B_{2}\left( x\right) \text{, for all }x\in \overline{B}_{R},
\label{vb}
\end{equation}%
where the inequalities serve as a foundational guideline. In cases where
these conditions are violated, the initial data must undergo adjustments to
ensure the value functions align with (\ref{vb}). With the parameters
explicitly defined, updating the Python code becomes easier, we only need to
adjust these values directly when model parameters change, rather than
re-running iterative numerical solvers.
\end{remark}

\section{Some Future Directions\label{7}}

This study introduces a stochastic production planning framework with regime
switching and provides qualitative insights into the sensitivity and model
analysis.

Future research directions include exploring alternative convex loss
functions, incorporating correlated Brownian motions, and modeling inventory
levels with geometric Brownian motion. Additionally, advancements in machine
learning could enable real-time detection of economic regimes, enhancing
adaptability and resilience.

While the current methodology is robust, there are several non-trivial
extensions worth exploring, along with potential approaches to tackle the
associated challenges:

\subsection{Alternative Convex Loss Functions}

One possible extension is to replace the quadratic loss function with other
convex loss functions, such as logarithmic or exponential functions, which
might better capture specific cost structures in certain industries.

Convex loss functions other than quadratic introduce non-linearities into
the Hamilton-Jacobi-Bellman (HJB) equations, complicating analytical and
numerical solutions.

\subsection{Correlated Brownian Motions}

The current model assumes independent Brownian motions driving the inventory
dynamics. An extension is to introduce correlations between these stochastic
processes to capture interdependencies among economic goods. Correlation
introduces mixed derivative terms in the PDE system, complicating analysis
and computation (see \cite{CDPCM}). Incorporating correlations makes the
model applicable to industries with interrelated goods or shared demand
drivers.

\subsection{Positivity with Geometric Brownian Motion}

An extension is to model inventory levels using geometric Brownian motion to
ensure positivity, which is essential in industries where negative inventory
levels are infeasible. Geometric Brownian motion introduces non-linear drift
and diffusion coefficients in the stochastic differential equations and HJB
system. Apply logarithmic transformations to convert multiplicative dynamics
into additive dynamics, simplifying the equations. Adaptive numerical
methods, such as mesh refinement techniques, can handle non-linear terms
effectively. This extension ensures realistic modeling of inventory levels
and enhances the applicability of the framework to various industries.

\section{Conclusion\label{8}}

The production planning problem is solved using a value function approach,
where the optimal production policy is characterized by a system of elliptic
PDEs. This paper aims to bridge the gap between theoretical modeling and
practical implementation, providing robust tools for stochastic production
planning under regime-switching parameters. The regime-switching framework
provides actionable insights for managerial decision-making, policy
analysis, and operational optimization.

The contributions of this study include: derivation of the
Hamilton-Jacobi-Bellman (HJB) equations and their transformation into an
elliptic PDE system, development of a monotone iteration scheme to
approximate solutions, enabling quantitative analysis of production
policies, investigation of the impacts of volatility, holding costs, and
discount rates on the value functions and comparison of models with and
without regime switching, highlighting the conservative and balanced
predictions of regime-switching models.

By adapting production strategies to economic cycles, minimizing costs, and
mitigating risks, the model enhances practical applicability in industries
such as automotive manufacturing, energy systems, and retail.

\section{Appendix\label{9}}

The numerical algorithm described in Section \ref{4} has been implemented
using Python, with assistance from Microsoft Edge Copilot, and is presented
below for the same data presented in Theorem \ref{s4}:

\bigskip 
\begin{lstlisting}
import numpy as np
import matplotlib.pyplot as plt
from scipy.optimize import fsolve

# Parameters
a1, alpha1 = 0.6, 0.3
a2, alpha2 = 0.9, 0.8
sigma1, sigma2 = 1, 0.3
M1, M2 = 5, 1
R = 10
epsilon = 1e-6
N = 100  # Number of grid points

# Step 1: Solve for K1 and K2 with automatic adjustment if needed
def solve_k_parameters(initial_guess):
    def equations(K):
        K1, K2 = K
        eq1 = 4 * K1**2 + (2 * (a1 + alpha1) * sigma1**2) / sigma1**4 * K1 - M1 / sigma1**4 - (2 * a1 * sigma2**2) / sigma1**4 * K2
        eq2 = 4 * K2**2 + (2 * (a2 + alpha2) * sigma2**2) / sigma2**4 * K2 - M2 / sigma2**4 - (2 * a2 * sigma1**2) / sigma2**4 * K1
        return [eq1, eq2]
    return fsolve(equations, initial_guess)

initial_guess = [-1, -1]
K1, K2 = solve_k_parameters(initial_guess)

# Loop until both K1 and K2 are negative
while K1 >= 0 or K2 >= 0:
    print(f"Incorrect values: K1 = {K1}, K2 = {K2}. Retesting with different estimates.")
    # Decrement the initial estimate to search for negative solutions
    initial_guess = [initial_guess[0] - 1, initial_guess[1] - 1]
    K1, K2 = solve_k_parameters(initial_guess)

# Additional calculations for S1 and S2
S1, S2 = np.exp(K1 * R**2), np.exp(K2 * R**2)


# Print K1 and K2 values
print(f"K1 = {K1}")
print(f"K2 = {K2}")

# Step 2: Define g1 and g2 functions
def g1(x, t, s):
    return (5*x**2 / sigma1**4) * t + (2 * (a1 + alpha1) / sigma1**2) * t * np.log(t) - (2 * a1 * sigma2**2 / sigma1**4) * t * np.log(s)

def g2(x, t, s):
    return (x**2 / sigma2**4) * s + (2 * (a2 + alpha2) / sigma2**2) * s * np.log(s) - (2 * a2 * sigma1**2 / sigma2**4) * s * np.log(t)

# Step 3: Compute Lambda1 and Lambda2
def g1_partial_t(x, t, s):
    return (5*x**2 / sigma1**4) + (2 * (a1 + alpha1) / sigma1**2) * (1 + np.log(t)) - (2 * a1 * sigma2**2 / sigma1**4) * np.log(s)

def g2_partial_s(x, t, s):
    return (x**2 / sigma2**4) + (2 * (a2 + alpha2) / sigma2**2) * (1 + np.log(s)) - (2 * a2 * sigma1**2 / sigma2**4) * np.log(t)

x_values = np.linspace(-R, R, 200)
Lambda1 = float("-inf")
Lambda2 = float("-inf")

for xi in x_values:
    for t_val in np.linspace(np.exp(K1 * (R**2 - xi**2)), 1, 50):
        for s_val in np.linspace(np.exp(K2 * (R**2 - xi**2)), 1, 50):
            value1 = g1_partial_t(xi, t_val, s_val)
            Lambda1 = max(Lambda1, value1)
            value2 = g2_partial_s(xi, t_val, s_val)
            Lambda2 = max(Lambda2, value2)

Lambda1 = -Lambda1
Lambda2 = -Lambda2

print(f"Lambda1 = {Lambda1}")
print(f"Lambda2 = {Lambda2}")

# Step 4: Successive approximation method using Finite Differences
def solve_system(u1, u2, Lambda1, Lambda2, max_iter=1000):
    for _ in range(max_iter):
        u1_old, u2_old = u1.copy(), u2.copy()
        for i in range(1, N-1):
            u1[i] = (g1(x[i], u1_old[i], u2_old[i]) + Lambda1 * u1_old[i] - (u1[i-1] + u1[i+1]) / (dx**2)) / (-2 / dx**2 + Lambda1)
            u2[i] = (g2(x[i], u1_old[i], u2_old[i]) + Lambda2 * u2_old[i] - (u2[i-1] + u2[i+1]) / (dx**2)) / (-2 / dx**2 + Lambda2)
        # Check for convergence
        if np.max(np.abs(u1 - u1_old)) < epsilon and np.max(np.abs(u2 - u2_old)) < epsilon:
            break
    return u1, u2

# Discretize the domain
x = np.linspace(-R, R, N)
dx = x[1] - x[0]

# Initialize u1 and u2
u1 = np.ones(N)
u2 = np.ones(N)

u1, u2 = solve_system(u1, u2, Lambda1, Lambda2)

# Value functions
z1 = -2 * sigma1**2 * np.log(u1)
z2 = -2 * sigma2**2 * np.log(u2)

# Optimal production rates
p1_star = -0.5 * np.gradient(z1, dx)
p2_star = -0.5 * np.gradient(z2, dx)

# Upper-Bounded Functions
B1 = -2 * sigma1**2 * K1 * (R**2 - x**2)
B2 = -2 * sigma2**2 * K2 * (R**2 - x**2)

# Plot the results
plt.figure(figsize=(10, 8))

# Plot z1(x) and z2(x) along with B1(x) and B2(x)
plt.subplot(2, 1, 1)
plt.plot(x, z1, label="z1(x)", color="blue")
plt.plot(x, z2, label="z2(x)", color="green")
plt.plot(x, B1, label="B1(x)", color="orange", linestyle="--")
plt.plot(x, B2, label="B2(x)", color="purple", linestyle="--")
plt.xlabel("x")
plt.ylabel("Functions")
plt.legend()
plt.grid()

# Plot p1*(x, 1) and p2*(x, 2)
plt.subplot(2, 1, 2)
plt.plot(x, p1_star, label="p1*(x, 1)", color="red")
plt.plot(x, p2_star, label="p2*(x, 2)", color="brown")
plt.xlabel("x")
plt.ylabel("Optimal Production Rates")
plt.legend()
plt.grid()

plt.tight_layout()
plt.show()

# -------------------------------------------------------------------
# Simulation of Inventory Dynamics (Euler-Maruyama)
# -------------------------------------------------------------------
dt = 0.01       # time step delta(t)
T_max = 100     # maximum simulation time
x0 = 0          # initial inventory: y(0) = x0
t = 0.0
y = x0
epsilon_regime = 1  # initial economic regime

# To store trajectory: list of tuples (t, y, epsilon)
inventory_traj = [(t, y, epsilon_regime)]

# Function to interpolate the optimal production rate based on the computed arrays
def optimal_production(y_value, regime):
    idx = np.argmin(np.abs(x - y_value))
    if regime == 1:
        return p1_star[idx]
    elif regime == 2:
        return p2_star[idx]
    else:
        raise ValueError("Regime unknown.")

# Function to update the economic regime based on transition probabilities
def update_regime(current_regime):
    r = np.random.rand()
    if current_regime == 1 and r < a1 * dt:
        return 2
    elif current_regime == 2 and r < a2 * dt:
        return 1
    return current_regime

# Euler-Maruyama simulation loop until |y(t)| >= R or T_max is reached
while np.abs(y) < R and t < T_max:
    epsilon_regime = update_regime(epsilon_regime)
    p_star = optimal_production(y, epsilon_regime)
    noise = np.random.normal(0, np.sqrt(dt))
    sigma_eps = sigma1 if epsilon_regime == 1 else sigma2
    y = y + p_star * dt + sigma_eps * noise
    t = t + dt
    inventory_traj.append((t, y, epsilon_regime))

inventory_traj = np.array(inventory_traj)
times = inventory_traj[:, 0]
inventory_levels = inventory_traj[:, 1]
regimes = inventory_traj[:, 2]

# Plot inventory level over time
plt.figure(figsize=(8, 4))
plt.plot(times, inventory_levels, label="Inventory level y(t)", color="blue")
plt.xlabel("Time")
plt.ylabel("Inventory Level y")
plt.title("Simulation of Inventory Dynamics")
plt.axhline(R, color='red', linestyle='--', label='Upper boundary R')
plt.axhline(-R, color='red', linestyle='--', label='Lower boundary -R')
plt.legend()
plt.grid()
plt.show()

print("Simulation ended at t =", t, "with y =", y, "and regime =", epsilon_regime)

\end{lstlisting}

\section{Acknowledgment.}

During the review process of the paper \cite{CCP}, the reviewers inquired
whether the obtained results could be implemented. At the time, we were
unable to provide a positive response to this issue. However, over time, the
question became a challenge that motivated me, and I have recently managed
to implement the results. I am grateful to the anonymous reviewer of the
paper \cite{CCP} for encouraging me to find a solution to the raised
question.

\section{Declarations}

\subparagraph{\textbf{Conflict of interest}}

The authors have no Conflict of interest to declare that are relevant to the
content of this article.

\subparagraph{\textbf{Ethical statement}}

The paper reflects the authors' original research, which has not been
previously published or is currently being considered for publication
elsewhere.

\bibliographystyle{plain} 
\bibliography{bibliography}
\end{document}